\documentclass[11pt]{article}
\usepackage[colorlinks=false]{hyperref}
\usepackage{amsfonts, amsmath, amssymb, amsthm, enumitem}

\author{Boris Bukh\thanks{Department of Mathematical Sciences, Carnegie Mellon University, Pittsburgh, PA 15213, USA. Partly supported by a Sloan Research Fellowship.}${\ }^{,}$\footnote{Supported in part by U.S. taxpayers through NSF grant DMS-1301548.} \and Zilin Jiang\thanks{Department of Mathematics, the Technion -- Israel Institute of Technology, Technion City, Haifa 3200003.}${\ }^{,}$\footnotemark[2]}
\title{Bipartite algebraic graphs without quadrilaterals}
\date{}

\DeclareMathOperator{\ex}{ex}                                   

\newtheorem{definition}{Definition}
\newtheorem{theorem}{Theorem}
\newtheorem{proposition}[theorem]{Proposition}
\newtheorem*{informal_conjecture*}{Informal conjecture}
\newtheorem{conjecture}{Conjecture}
\newtheorem{false_conjecture}[conjecture]{False Conjecture}

\newtheorem{lemma}[theorem]{Lemma}
\newtheorem{corollary}[theorem]{Corollary}
\theoremstyle{remark}
\newtheorem{remark}{Remark}
\newtheorem{example}{Example}

\newcommand{\NN}{\mathbb{N}}

\newcommand{\CC}{\mathbb{C}}
\newcommand{\FF}{\mathbb{F}}
\newcommand{\KK}{\mathbb{K}}
\newcommand{\PP}{\mathbb{P}}
\renewcommand{\AA}{\mathbb{A}}
\newcommand{\cc}{{\overline{c}}}
\newcommand{\dd}{\overline{d}}
\newcommand{\xx}{{\overline{x}}}
\newcommand{\yy}{\overline{y}}
\newcommand{\uu}{\overline{u}}
\newcommand{\vv}{\overline{v}}
\newcommand{\from}{\colon}
\newcommand{\variety}[1]{\left\{{#1}\right\}}
\newcommand{\rto}{\dashrightarrow}
\newcommand{\CChom}{\CC_{\mathrm{hom}}}
\newcommand{\KKhom}{\KK_{\mathrm{hom}}}
\newcommand{\aut}[1]{\operatorname{Cr}\left(#1\right)}
\newcommand{\autreg}[1]{\operatorname{Aut}\left(#1\right)}
\newcommand{\abs}[1]{\lvert #1\rvert}
\newcommand{\set}[2]{\left\{#1 : #2\right\}}
\newcommand{\id}{\operatorname{id}}
\newcommand{\gcdp}[1]{\gcd\left({#1}\right)}

\newcommand{\Lr}{\mathcal{L}_r}

\begin{document}
\maketitle

\begin{abstract}
  Let $\mathbb{P}^s$ be the $s$-dimensional complex projective space, and let $X, Y$ be two non-empty open subsets of $\mathbb{P}^s$ in the Zariski topology. A hypersurface $H$ in $\mathbb{P}^s\times\mathbb{P}^s$ induces a bipartite graph $G$ as follows: the partite sets of $G$ are $X$ and $Y$, and the edge set is defined by $\overline{u}\sim\overline{v}$ if and only if $(\overline{u},\overline{v})\in H$. Motivated by the Tur\'an problem for bipartite graphs, we say that $H\cap (X\times Y)$ is $(s,t)$-grid-free provided that $G$ contains no complete bipartite subgraph that has $s$ vertices in $X$ and $t$ vertices in $Y$. We conjecture that every $(s,t)$-grid-free hypersurface is equivalent, in a suitable sense, to a hypersurface whose degree in $\overline{y}$ is bounded by a constant $d = d(s,t)$, and we discuss possible notions of the equivalence.
  
  We establish the result that if $H\cap(X\times \mathbb{P}^2)$ is $(2,2)$-grid-free, then there exists $F\in \mathbb{C}[\overline{x},\overline{y}]$ of degree $\le 2$ in $\overline{y}$ such that $H\cap(X\times \mathbb{P}^2) = \{F = 0\}\cap (X\times \mathbb{P}^2)$. Finally, we transfer the result to algebraically closed fields of large characteristic.
\end{abstract}

\section{Introduction}

  The Tur\'{a}n number $\ex(n, F)$ is the maximum number of edges in an $F$-free graph\footnote{We say a graph is $F$-free if it does not have a subgraph isomorphic to $F$.} on $n$ vertices. The first systematic study of $\ex(n, F)$ was initiated by Tur\'{a}n \cite{turan1941extremal}, who solved the case when $F = K_t$ is a complete graph on $t$ vertices. Tur\'{a}n's theorem states that, on a given vertex set, the $K_t$-free graph with the most edges is the complete and balanced $(t-1)$-partite graph, in that the part sizes are as equal as possible.
  
  For general graphs $F$, we still do not know how to compute the Tur\'{a}n number exactly, but if we are satisfied with an approximate answer, the theory becomes quite simple: Erd\H{o}s and Stone \cite{MR0018807} showed that if the chromatic number $\chi(F)=t$, then $\ex(n, F) = \ex(n, K_t) + o(n^2) = \left(1-\frac{1}{t-1}\right){n\choose 2}+o(n^2)$. When $F$ is not bipartite, this gives an asymptotic result for the Tur\'{a}n number. On the other hand, for all but few bipartite graphs $F$, the order of $\ex(n,F)$ is not known. Most of the research on this problem focused on two classes of graphs: complete bipartite graphs and cycles of even length. A comprehensive survey is given by F\"{u}redi and Simonovits \cite{MR3203598}. 
  
  Suppose $G$ is a $K_{s,t}$-free graph with $s \le t$. The K\"{o}vari--S\'{o}s--Tur\'{a}n theorem \cite{MR0065617} implies an upper bound $\ex(n, K_{s,t}) \le \frac{1}{2}\sqrt[s]{t-1}\cdot n^{2-1/s}+o(n^{2-1/s})$, which was improved by F\"{u}redi \cite{MR1395691} to 
  $$\ex(n, K_{s,t}) \le \frac{1}{2}\sqrt[s]{t-s+1}\cdot n^{2-1/s} + o(n^{2-1/s}).$$
  Despite the lack of progress on the Tur\'{a}n problem for complete bipartite graphs, there are certain complete bipartite graphs for which the problem has been solved asymptotically, or even exactly. The constructions that match the upper bounds in these cases are all similar to one another. Each of the constructions is a bipartite graph $G$ based on an algebraic hypersurface\footnote{An algebraic hypersurface in a space of dimension $n$ is an algebraic subvariety of dimension $n-1$. The terminology from algebraic geometry used throughout the article is standard, and can be found in \cite{MR3100243}.} $H$. Both partite sets of $G$ are $\FF_p^s$ and the edge set is defined by: $\uu\sim\vv$ if and only if $(\uu, \vv)\in H$. In short, $G = \left(\FF_p^s, \FF_p^s, H(\FF_p)\right)$, where $H(\FF_p)$ denotes the $\FF_p$-points of $H$. Note that $G$ has $n := 2p^s$ vertices.
  
  In the previous works of Erd\H{o}s, R\'{e}nyi and S\'{o}s \cite{MR0223262}, Brown \cite{MR0200182}, F\"{u}redi \cite{MR1395763}, Koll\'{a}r, R\'{o}nyai and Szab\'{o} \cite{MR1417348} and Alon, R\'{o}nyai and Szab\'{o} \cite{MR1699238}, various hypersurfaces were used to define $K_{s,t}$-free graphs. Their equations were \begin{subequations}\label{eq1}
    \begin{align}
      x_1y_1 + x_2y_2 = 1, & \quad \text{for }K_{2,2}; \label{eq1a} \\
      (x_1-y_1)^2 + (x_2-y_2)^2 + (x_3-y_3)^2 = 1, & \quad \text{for }K_{3,3}; \label{eq1b} \\
      (N_s\circ\pi_s)(x_1 + y_1, x_2 + y_2, \dots, x_s + y_s) = 1, & \quad \text{for }K_{s,t} \text{ with }t \ge s! + 1; \label{eq1c} \\
      (N_{s-1}\circ\pi_{s-1})(x_2 + y_2, x_3 + y_3, \dots, x_s + y_s) = x_1y_1, & \quad \text{for }K_{s,t} \text{ with }t \ge (s-1)! + 1, \label{eq1d}
    \end{align}
  \end{subequations}
  where $\pi_s\from \FF_p^s \to \FF_{p^s}$ is an $\FF_p$-linear isomorphism and $N_s(\alpha)$ is the field norm, $N_s(\alpha) :=\alpha^{(p^s-1)/(p-1)}$.
  
  Clearly, the coefficients in \eqref{eq1a} and \eqref{eq1b} are integers and even independent of $p$. With some work, one can show that both \eqref{eq1c} and \eqref{eq1d} are polynomial equations of degree $\le s$ with coefficients in $\FF_p$. Therefore each equation in \eqref{eq1} can be written as $F(\xx, \yy) := F(x_1, \dots, x_s, y_1, \dots, y_s) = 0$ for some $F(\xx,\yy)\in \FF_p[\xx,\yy]$ of bounded degree. The previous works directly count the number of $\FF_p$ solutions to $F(\xx,\yy)=0$ and yield $\abs{H(\FF_p)} = \Theta(p^{2s-1}) = \Theta(n^{2-1/s})$, for each prime\footnote{We need $p \equiv 3\pmod4$ for \eqref{eq1b} to get the correct number of $\FF_p$ points on $H$. If $p \equiv 1\pmod4$, then the right hand side of \eqref{eq1b} should be replaced by a quadratic non-residue in $\FF_p$.}~$p$.

  \begin{definition}\label{kst}
    Given two sets $P_1$ and $P_2$, a set $V\subset P_1\times P_2$ is said to contain an $(s,t)$-grid if there exist $S \subset P_1, T\subset P_2$ such that $s=\abs{S}$, $t=\abs{T}$ and $S\times T\subset V$. Otherwise, we say that $V$ is $(s,t)$-grid-free.
  \end{definition}

  Observe that every $F(\xx,\yy)$ derived from \eqref{eq1} is symmetric in $x_i$ and $y_i$ for all $i$. We know that $(\uu,\vv)\in H$ if and only if $(\vv, \uu)\in H$ for all $\uu,\vv\in\FF_p^s$. The resulting bipartite graph $G = \left(\FF_p^s, \FF_p^s, H(\FF_p)\right)$ would be an extremal $K_{s,t}$-free graph if $H(\FF_p)$ had been $(s,t)$-grid-free.
  
  So which graphs are $K_{s,t}$-free with a maximum number of edges? The question was considered by Zolt\'{a}n F\"{u}redi in his unpublished manuscript \cite{furedi1988quadrilateral} asserting that every $K_{2,2}$-free graph with $q$ vertices (for $q \ge q_0$) and $\frac{1}{2}q(q+1)^2$ edges is obtained from a projective plane via a polarity with $q+1$ absolute elements. This loosely amounts to saying that all extremal $K_{2,2}$-free graphs are defined by generalization of \eqref{eq1a}.
  
  However, classification of all extremal $K_{s,t}$-free graphs seems out of reach. We restrict our attention to algebraically constructed graphs. Given a field $\FF$ and a hypersurface $H$ defined over $\FF$, it is natural to ask when $H(\FF)$ is $(s,t)$-grid-free. Because the general case is difficult, we work with algebraically closed fields $\KK$ in this paper. Denote by $\PP^s(\KK)$ the $s$-dimensional projective space over $\KK$. We are interested in hypersurface $H$ in $\PP^s(\KK)\times\PP^s(\KK)$.
  
  Since standard machinery from model theory, to be discussed in Section~\ref{sec_finite_char}, allows us to transfer certain results over $\CC$ (the field of complex numbers) to algebraically closed fields of large characteristic, our focus will be on the $\KK = \CC$ case. We use $\PP^s$ for the $s$-dimensional complex projective space and $\AA^s := \PP^s \setminus \variety{x_0 = 0}$ for the $s$-dimensional complex affine space.

  Note that even if $H$ contains $(s,t)$-grids, one may remove a few points from the projective space to destroy all $(s,t)$-grids in $H$. For example, the homogenization of \eqref{eq1b} is
  $$(x_1y_0 - x_0y_1)^2 + (x_2y_0 - x_0y_2)^2 + (x_3y_0 - x_0y_3)^2 = x_0^2y_0^2.$$
  The equation defines hypersurface $H$ in $\PP^3\times \PP^3$. Let $V := \{x_0 = x_1^2 + x_2^2 + x_3^2 = 0\}$ be a variety in $\PP^3$. Since $V\times \PP^3\subset H$, $H$ contains a lot of $(3,3)$-grids. However, $H\cap(\AA^3\times\AA^3)$ is $(3,3)$-grid-free. 
  \begin{definition}\label{def:almost}
    A set $V\subset \PP^s\times \PP^s$ is almost-$(s,t)$-grid-free if there are two nonempty Zariski-open sets $X, Y\subset\PP^s$ such that $V\cap(X\times Y)$ is $(s,t)$-grid-free.
  \end{definition}

  Suppose the defining equation of $H$, say $F(\xx,\yy)$, is of low degree in $\yy$. Heuristically, for generic\footnote{Henceforth, a statement is true for a generic point $\uu \in \PP^s$ means that there exists a nonempty Zariski-open set $U \subset \PP^s$ such that the statement is true for every $\uu \in U$.} distinct $\uu_1, \dots, \uu_s\in\PP^s$, by B\'{e}zout's theorem, one would expect $\{F(\uu_1,\yy) = \dots = F(\uu_s,\yy) = 0\}$ to have few points. So we conjecture the following.
  \begin{informal_conjecture*}
    Every almost-$({s,t})$-grid-free hypersurface is equivalent, in a suitable sense, to a hypersurface whose degree in $\yy$ is bounded by some constant $d := d(s,t)$.
  \end{informal_conjecture*}
  The right equivalence notion depends on $X$ and $Y$ in Definition~\ref{def:almost}. We shall discuss possible notions of equivalence in Section~\ref{sec_conj}, and make three specific conjectures. Results in support of these conjectures can be found in Section \ref{sec_s=1} and Section \ref{sec_partial}. 
  
  Before we make our conjectures precise, we note that an analogous situation occurs for $C_{2t}$-free graphs. The upper bound $\ex(n, C_{2t}) = O(n^{1+1/t})$ first established by Bondy--Simonovits \cite{MR0340095} has been matched only for $t=2,3,5$. The $t=2$ case was already mentioned above because $C_4 = K_{2,2}$. The constructions for $t=3, 5$ are also algebraic (see \cite{MR0197342,MR2227729} for $t = 3$ and \cite{MR0197342,MR1109426} for $t=5$). Also, a conjecture in a similar spirit about algebraic graphs of girth eight was made by Dmytrenko, Lazebnik and Williford \cite{MR2359323}. It was recently resolved by Hou, Lappano and Lazebnik \cite{MR3575461}.

  The paper is organized as follows. In Section \ref{sec_conj} we flesh out the informal conjecture above, in Section~\ref{sec_s=1} we briefly discuss the $s=1$ case, in Section~\ref{sec_partial} we partially resolve the $s=t=2$ case, and finally in Section \ref{sec_finite_char}, we consider algebraically closed fields of large characteristic.
  
\section{Conjectures on the $({s,t})$-grid-free case}\label{sec_conj}
  
  Given a field $\FF$, we denote by $\FF_{\mathrm{hom}}[\xx]$ the set of homogeneous polynomials in $\FF[\xx]$ and by $\FF_{\mathrm{hom}}[\xx,\yy]$ the set of polynomials in $\FF[\xx,\yy]$ that are separately homogeneous in $\xx$ and $\yy$.
  
  We might be tempted to guess the following instance of the informal conjecture.
  \begin{false_conjecture}\label{false}
    If $H$ is almost-$({s,t})$-grid-free, then there exists $F(\xx,\yy)\in\CChom[\xx,\yy]$ of degree $\le d$ in $\yy$ for some $d = d(s,t)$ such that $H = \variety{F=0}$.
  \end{false_conjecture}
  
  Unfortunately, Conjecture \ref{false} is false because of the following example.
  
  \begin{example}\label{eg.line}
    Consider $H_0:=\{x_0y_0 + x_1y_1 + x_2y_2 = 0\}$ and $H_1$ defined by \begin{equation}\label{arb.large}
    	x_0y_0^d +x_1y_0^{d-1}y_1 + x_2\left(y_0^{d-1}y_2+y_0^df(y_1/y_0)\right) = 0,
    \end{equation} where $f$ is a polynomial of degree $d$. One can check that both $H_0$ and $H_1\setminus\{y_0=0\}$ are $({2,2})$-grid-free, whereas equation \eqref{arb.large} can be of arbitrary large degree in $\yy$.
  \end{example}
  
  Behind Example \ref{eg.line} is the birational automorphism $\sigma\from \PP^2\rto\PP^2$ defined by $$\sigma(y_0:y_1:y_2):=\left(y_0^d: y_0^{d-1}y_1:y_0^{d-1}y_2+y_0^df(y_1/y_0)\right).$$ Note that $\id\times\sigma$ is a biregular map\footnote{A biregular map is a regular map whose inverse is also regular.} from $H_1\setminus\{y_0=0\}$ to $H_0\setminus\{y_0=0\}$. Composition with the automorphism increased the degree of $H_0$ in $\yy$ while preserving almost-$(2,2)$-grid-freeness. 
  Here is another example illustrating the relationship between birational automorphisms and $({s,t})$-grid-free hypersurfaces.
  
  \begin{example}
    Define $H_2 := \{x_0y_1y_2 + x_1y_0y_2 + x_2y_0y_1 = 0\}$. One can also check that $H_2\setminus \variety{y_0y_1y_2=0}$ is $(2,2)$-grid-free. Behind this example is \emph{the standard quadratic transformation} $\sigma$ from $\PP^2$ to itself given by $\sigma(y_0:y_1:y_2)=(y_1y_2:y_0y_2:y_0y_1)$. Note that $\id\times \sigma$ is a biregular map from $H_2\setminus\variety{y_0y_1y_2=0}$ to $H_0\setminus\variety{y_0y_1y_2=0}$.
  \end{example}
  
  Let $\aut{\PP^s}$ be the group of birational automorphisms on $\PP^s$, also known as \emph{the Cremona group}. Evidently, the almost-$({s,t})$-grid-freeness is invariant under $\aut{\PP^s}\times \aut{\PP^s}$.
  
  \begin{proposition}
    If $V_1\subset\PP^s\times\PP^s$ is an almost-$(s,t)$-grid-free set, then so is $V_2 := (\sigma_X\times\sigma_Y)V_1$ for all $\sigma_X, \sigma_Y\in\aut{\PP^s}$.
  \end{proposition}
  
  \begin{proof}
    Suppose $\sigma_X\from X_1' \to X_2', \sigma_Y\from Y_1' \to Y_2'$ are isomorphisms respectively, where $X_1', X_2', Y_1', Y_2'$ are nonempty Zariski-open subsets of $\PP^s$. There are nonempty open subsets $X_1\subset X_1', Y_1\subset Y_1'$ of $\PP^s$, such that $V_1\cap(X_1\times Y_1)$ is $(s,t)$-grid-free.  Now we know that
    $$V_2\cap \left(\sigma_X\left(X_1\right)\times \sigma_Y\left(Y_1\right)\right) = (\sigma_X\times\sigma_Y)\left(V_1\cap \left(X_1\times Y_1\right)\right)$$ is $(s,t)$-grid-free too and that $\sigma_X(X_1),\sigma_Y(Y_1)$ are nonempty and Zariski-open.
  \end{proof}

  \begin{remark}
  Though little is known about the structure of the Cremona group in 3 dimensions and higher, the classical Noether--Castelnuovo theorem says that the Cremona group $\aut{\PP^2}$ is generated by the group of projective linear transformations and the standard quadratic transformation. The proof of this theorem, which is very delicate, can be found in \cite[Chapter~8]{MR1874328}.
  \end{remark}
  
  We say that two sets $V_1,V_2\subset\PP^s\times\PP^s$ are \emph{almost equal} if there exist nonempty Zariski-open sets $X, Y\subset\PP^s$ such that $V_1\cap(X\times Y) = V_2\cap(X\times Y)$. We believe that the only obstruction to Conjecture~\ref{false} is the Cremona group.

  \begin{conjecture}\label{conj}
    Suppose $H$ is a hypersurface in $\PP^s\times\PP^s$. If $H$ is almost-$({s,t})$-grid-free, then there exist $\sigma \in \aut{\PP^s}$ and $F(\xx,\yy)\in\CChom[\xx,\yy]$ of degree $\le d$ in $\yy$ for some $d = d(s,t)$ such that $H$ is almost equal to $\variety{F\circ(\id\times\sigma)=0}$.
  \end{conjecture}
  
  In fact, we believe in an even stronger conjecture.
  
  \begin{conjecture}\label{strong}
  	Suppose $H$ is a hypersurface in $\PP^s\times\PP^s$. Let $X, Y$ be nonempty Zariski-open subsets of $\PP^s$. If $H\cap (X\times Y)$ is $(s,t)$-grid-free, then there exist $Y'\subset \PP^s$, a biregular map $\sigma\from Y \to Y'$ and $F(\xx,\yy)\in\CChom[\xx,\yy]$ of degree $\le d$ in $\yy$ for some $d = d(s,t)$ such that $H\cap(X\times Y) = \variety{F\circ(\id\times\sigma)=0}\cap(X\times Y)$.
  \end{conjecture}
  
  We prove Conjecture~\ref{strong} if $s = 1$ and if $s = t = 2, Y = \PP^2$ (see Section~\ref{sec_s=1} and \ref{sec_partial} respectively).
  
  One special case is when $H\cap(\AA^s\times \AA^s)$ is $(s,t)$-grid-free. In this case, $H$ can be seen as an affine algebraic hypersurface in $2s$-dimensional affine space. The group of automorphisms of $\AA^s$, denoted by $\autreg{\AA^s}$, is a subgroup of the Cremona group. In this special case, we make a stronger conjecture.
  
  \begin{conjecture}\label{conjreg}
    Suppose $H$ is an affine hypersurface in $\AA^s\times\AA^s$. If $H$ is $({s,t})$-grid-free, then there exist $\sigma \in \autreg{\AA^s}$ and $F(\xx,\yy)\in\CC[\xx,\yy]$ of degree $\le d$ in $\yy$ for some $d = d(s,t)$ such that $H = \{F\circ(\id\times\sigma) = 0\}$.
  \end{conjecture}
  
  \begin{remark}
   An automorphism $\sigma\in\autreg{\AA^s}$ is \emph{elementary} if it has a form \[
    \sigma\from (x_1, \dots, x_{i-1},x_i,x_{i+1},\dots, x_s) \mapsto (x_1, \dots, x_{i-1},cx_i + f, x_{i+1},\dots, x_s),
    \] where $0 \neq c \in \CC, f\in \CC[x_1, \dots, x_{i-1},x_{i+1},\dots, x_s]$. In Example \ref{eg.line}, we used the homogenization of an elementary automorphism to make a counterexample to Conjecture \ref{false}. The \emph{tame subgroup} is the subgroup of $\autreg{\AA^s}$ generated by all the elementary automorphisms, and the elements from this subgroup are called \emph{tame automorphisms}, while non-tame automorphisms are called \emph{wild}. It is known \cite{MR0008915,MR0054574} that all the elements of $\autreg{\AA^2}$ are tame. However, in the case of 3 dimensions, the following automorphism constructed by Nagata (see \cite{nagata1972automorphism}):
  \[
    \sigma(x,y,z) = \left(x+(x^2-yz)z, y+2(x^2-yz)x+(x^2-yz)z,z\right)
  \] was shown \cite{MR2017754,MR2015334} to be wild. See also \cite{MR2654304}. We note that the question on the existence of wild automorphisms remains open for higher dimensions.    
  \end{remark}
    
\section{Results on the $(1,t)$-grid-free case}\label{sec_s=1}
  
  As for the $s=1$ case, one is able to fully characterize $(1,t)$-grid-free hypersurfaces. We always assume that $H$ is a hypersurface in $\PP^1\times\PP^1$ and $X, Y$ are nonempty Zariski-open subsets of $\PP^1$ throughout this section.
  
  \begin{theorem}\label{s=1}
    Suppose $H= \{\tilde{F}=0\}$, where $\tilde{F}(\xx,\yy)\in\CChom[\xx,\yy]$. Let
    \begin{equation}\label{factorF}
      \tilde{F}(\xx,\yy)=f(\xx)g(\yy)h_1(\xx,\yy)^{r_1}h_2(\xx,\yy)^{r_2}\dots h_n(\xx,\yy)^{r_n}
    \end{equation} be the factorization of $\tilde{F}$ such that $h_1, h_2, \dots, h_n$ are distinct irreducible polynomials depending on both $\xx$ and $\yy$. Let $d_i$ be the degree of $h_i$ in $\yy$. Then $H\cap(X\times Y)$ is $(1,t)$-grid-free if and only if $\variety{f=0}\cap X = \emptyset$ and $\abs{\variety{g=0}\cap Y} + d_1 + d_2 + \dots + d_n < t$.
  \end{theorem}
  
  \begin{remark}
    In Theorem~\ref{s=1}, note that $f(\xx) \in \CChom[\xx]$, $g(\yy)\in \CChom[\yy]$ and $h_i(\xx,\yy)\in\CChom[\xx,\yy]$ for all $i = 1, 2, \dots, n$.
  \end{remark}
  
  \begin{proof}
    Clearly, if $H\cap(X\times Y)$ is $(1,t)$-grid-free, then $\variety{f=0}\cap X$ is empty. For every $\uu\in \PP^1$, consider the following $n+n+n+n+\binom{n}{2}$ systems of (one or two) inequalities and/or equations in $\yy$:
    \begin{subequations}
      \begin{align}
        \deg h_i(\uu,\yy) < d_i, & \quad i = 1, 2, \dots, n; \label{subeqa} \\
        \yy\in \PP^1\setminus Y\text{ and }h_i(\uu,\yy)= 0, & \quad i = 1, 2, \dots, n; \label{subeqb} \\
        h_i(\uu,\yy) = \partial_yh_i(\uu,\yy) = 0, & \quad i = 1, 2, \dots, n; \label{subeqc} \\
        h_i(\uu,\yy) = g(\yy) = 0, & \quad i = 1, 2, \dots, n; \label{subeqd} \\
        h_i(\uu,\yy) = h_j(\uu,\yy) = 0, & \quad  i \neq j. \label{subeqe}
      \end{align}      
    \end{subequations}
    
    We claim that each of these systems has no solution in $\PP^1$ for a generic $\uu$. Clearly, this is so for (\ref{subeqa}, \ref{subeqb}). For (\ref{subeqc}, \ref{subeqd}, \ref{subeqe}), the claim follows from the following fact: if $p_1(\xx, \yy), p_2(\xx, \yy) \in \CChom[\xx, \yy]$ are relatively prime in the unique factorization domain $\CChom[\xx,\yy]$, then 
    \begin{equation}\label{bezout}
      p_1(\uu, \yy) = p_2(\uu, \yy) = 0      
    \end{equation}
    has no solution in $\PP^1$ for a generic $\uu$. Indeed, notice that $p_1(\xx, \yy), p_2(\xx, \yy)$ are also relatively prime in the Euclidean domain $\CC(\xx)[\yy]$. By the Euclidean algorithm, there are $q_1(\xx, \yy), q_2(\xx, \yy) \in \CC(\xx)[\yy]$ such that $p_1q_1 + p_2q_2 = 1$. Let $q\in\CC[\xx]$ be the common denominator of $q_1$ and $q_2$. We then have the equality $p_1\tilde{q}_1 + p_2\tilde{q}_2 = q$, where $\tilde{q}_i = q_iq \in \CC[\xx][\yy]$ for $i \in [2]$. Thus for all $\uu$ such that $q(\uu)\neq 0$, \eqref{bezout} has no solution.
    
    So for a generic $\uu\in\PP^1$, $\tilde{F}(\uu,\yy) = 0$ has exactly $M := \abs{\variety{g=0}\cap Y} + d_1 + d_2 + \dots + d_n$ distinct solutions in $Y$. The conclusion follows as $M$ is the maximal number of distinct solutions.
  \end{proof}
  
  The informal conjecture thus holds when $s=1$ as Theorem \ref{s=1} implies:
  
  \begin{corollary}\label{c.s=1}
    If $H\cap (X\times Y)$ is $(1,t)$-grid-free, then there exists $F(\xx,\yy)\in\CChom[\xx,\yy]$ of degree $< t$ in $\yy$ such that $H\cap(X\times Y) = \variety{F=0}\cap(X\times Y)$.
  \end{corollary}
  
  \begin{proof}
    Let $H = \{\tilde{F} = 0\}$, and let $f, g$ and $h_1, h_2,\dots, h_n$ be the factors of $\tilde{F}$ as in \eqref{factorF}. Suppose $m := \abs{\variety{g=0}\cap Y}$ and $\variety{g=0}\cap Y = \{\vv_1, \vv_2,\dots,\vv_m\}$. Let $g_i(\yy)\in\CChom[\yy]$ be linear such that $\variety{g_i = 0}=\{\vv_i\}$.
    By Theorem~\ref{s=1}, $F(\xx,\yy) := g_1(\yy)g_2(\yy)\dots g_m(\yy)h_1(\xx,\yy)h_2(\xx,\yy)\dots h_n(\xx,\yy)$ is of degree $m + d_1 + d_2 + \dots d_n < t$ in $\yy$. Clearly, the zeros of $F$ in $X\times Y$ agree with $\tilde{F}$.
  \end{proof}
  
  Conjectures~\ref{conj}, \ref{strong} and \ref{conjreg} follow from the corollary in the $s=1$ case. The biregular map $\sigma$ becomes trivial in these conjectures since $\aut{\PP^1}$ consists only of projective linear transformations.
  
\section{Results on the $(2,2)$-grid-free case}\label{sec_partial}

Throughout the section we assume that $H$ is a hypersurface in $\PP^2\times\PP^2$ and $X$ is a nonempty Zariski-open subset of $\PP^2$.

  \begin{theorem}\label{main_partial_result}
    If $H\cap(X\times \PP^2)$ is $(2,2)$-grid-free, then there exists $F(\xx,\yy)\in\CChom[\xx,\yy]$ of degree $\le 2$ in $\yy$ such that $H\cap(X\times \PP^2) = \{F=0\}\cap(X\times \PP^2)$.
  \end{theorem}
  
  The theorem resolves Conjecture~\ref{strong} for $s = t = 2, Y = \PP^2$. Note that the biregular map $\sigma\from Y\to Y'$ in the conjecture does not appear in the theorem. The reason is as follows. In the restricted setting $Y = \PP^2$, since every biregular map defined on $\PP^2$ has to be a biregular automorphism of $\PP^2$ and every biregular automorphism of $\PP^2$ is a linear transformation, the biregular map $\sigma$ is linear, hence $F\circ(\id \times \sigma)$ would be of the same degree as $F$.
  
  Our argument uses a reduction to an intersection problem of plane algebraic curves. The key ingredient is a theorem by Moura \cite{MR2019943} on the intersection multiplicity of plane algebraic curves.

  \begin{theorem}[Moura \cite{MR2019943}]\label{moura}
    Denote by $I_{\vv}(C_1, C_2)$ the intersection multiplicity of algebraic curves $C_1$ and $C_2$ at $\vv$. 
    For a generic point $\vv$ on an irreducible algebraic curve $C_1$ of degree $d_1$ in $\PP^2$,
    \[
      \max_{C_2}\set{I_{\vv}(C_1, C_2)}{C_1 \not\subset C_2, \deg C_2 \le d_2} = \begin{cases}
        \frac{1}{2}(d_2^2+3d_2) & \text{if }d_1 > d_2; \\
        d_1d_2 - \frac{1}{2}(d_1^2-3d_1+2) & \text{if }d_1 \le d_2.
      \end{cases}
    \]
  \end{theorem}
  
  \begin{corollary}\label{cmoura}
    For a generic point $\vv$ on an algebraic curve $C$ in $\PP^2$, any algebraic curve $C'$ with $\vv\in C'$ intersects with $C$ at another point unless $C$ is irreducible of degree $\le 2$.
  \end{corollary}
  
  \begin{proof}
    Suppose $C$ has more than one irreducible components. Let $C_1$ and $C_2$ be any two of them. Since $C_1 \cap C_2$ is finite, we can pick a generic point $\vv$ on $C_1 \setminus C_2$. Now any algebraic curve $C'$ containing $\vv$ intersects $C$ at another point on $C_2$. So, $C$ is irreducible.
    
    Let $d$ and $d'$ be the degrees of $C$ and $C'$ respectively. By Theorem~\ref{moura}, one can check that $I_{\vv}(C, C') < dd'$ for a generic point $\vv\in C$ for all $d > 2$. Since B\'{e}zout's theorem states that the total intersection multiplicities of $C$ and $C'$ at their common points is equal to $dd'$, we deduce that $C$ intersects $C'$ at another point unless $d\le 2$.
  \end{proof}
  
  In our proof of Theorem~\ref{main_partial_result}, we think of $H$ as a family of algebraic curves in $\PP^2$, each of which is indexed by $\uu\in \PP^2$ and is defined by $C(\uu) := \set{\vv\in\PP^2}{(\uu,\vv)\in H}$. We call algebraic curve $C(\uu)$ the \emph{section} of $H$ at $\uu$. The set $H \cap (X\times \PP^2)$ is $(2,2)$-grid-free
if and only if $C(\uu)$ and $C(\uu')$ intersect in at most one point for all distinct $\uu, \uu'\in X$. 
  The last piece that we need for our proof is a technical lemma on generic sections of irreducible hypersurfaces.
  
  \begin{lemma}\label{generic_section}
    Suppose $H_1$ and $H_2$ are two different irreducible hypersurfaces in $\PP^2\times \PP^2$ defined by $h_1(\xx,\yy), h_2(\xx,\yy) \in \CChom[\xx,\yy]\setminus(\CChom[\xx]\cup \CChom[\yy])$ respectively. Denote the section of $H_i$ at $\uu$ by $C_i(\uu)$ for $i = 1,2$. For generic $\uu\in \PP^2$, $C_1(\uu)$ and $C_2(\uu)$ share no common irreducible components, and moreover, each $C_i(\uu)$ is a reduced\footnote{The algebraic curve $C_i(\uu)$ is reduced in the sense that its defining equation $h_i(\uu, \yy)$ is square-free.} algebraic curve.
  \end{lemma}

  \begin{proof}[Proof of Theorem~\ref{main_partial_result} assuming Lemma~\ref{generic_section}]
    Suppose $H\cap(X\times \PP^2)$ is $(2,2)$-grid-free. Take an arbitrary $\uu\in X$ and consider algebraic curve $C(\uu)$ in $\PP^2$. We claim that every $\vv\in C(\uu)$ is an intersection of $C(\uu)$ and $C(\uu')$ for some $\uu'\in X\setminus\{\uu\}$. Define $D(\vv) := \set{\uu \in X}{(\uu, \vv)\in H}$. Since $\PP^2\setminus X$ is Zariski-closed, the set $D(\vv)$ is either empty or infinite.  However, $\uu\in D(\vv)$ and the claim is equivalent to $\abs{D(\vv)} \ge 2$.
    
    Now pick a generic $\vv\in C(\uu)$. We know that point $\vv$ is an intersection of $C(\uu)$ and $C(\uu')$ for some $\uu'\in X\setminus \{\uu\}$ and it is the only intersection because $H\cap(X\times \PP^2)$ is $(2,2)$-grid-free. We apply Corollary~\ref{cmoura} to $C(\uu)$ and $C(\uu')$ and get that $C(\uu)$ is irreducible of degree $\le 2$.
    
    Suppose $H$ is defined by $\tilde{F}(\xx,\yy)\in\CChom[\xx,\yy]$ and 
    \begin{equation}\label{factor2}
      \tilde{F}(\xx,\yy) = f(\xx)g(\yy)h_1(\xx,\yy)^{r_1}h_2(\xx,\yy)^{r_2}\dots h_n(\xx,\yy)^{r_n}
    \end{equation}
    is the factorization of $\tilde{F}$ such that $h_1, h_2, \dots, h_n$ are distinct irreducible polynomials in $\CChom[\xx,\yy]\setminus(\CChom[\xx]\cup\CChom[\yy])$. The set $\{f=0\}\cap X$ is either empty or infinite. So, for $H\cap(X\times \PP^2)$ to be $(2,2)$-grid-free we must have $\{f=0\}\cap X=\emptyset$. Similarly, we know that $\variety{g = 0} = \emptyset$, that is, $g(\yy)$ is a nonzero constant. These imply that the zeros of $F(\xx, \yy) := h_1(\xx, \yy)h_2(\xx, \yy)\dots h_n(\xx,\yy)$ in $X\times \PP^2$ agree with $\tilde{F}$.

    Let $C_i(\uu)$ be the section of $H_i := \{h_i = 0\}$ at $\uu$ for $i = 1, 2, \dots, n$. From Lemma~\ref{generic_section}, we know that, for a generic $\uu\in \PP^2$, $C_i(\uu)$ and $C_j(\uu)$ have no common irreducible components for all $i\neq j$. Therefore $C(\uu) = \cup_{i=1}^n C_i(\uu)$ has at least $n$ irreducible components, and so $n = 1$. Now $C(\uu) = C_1(\uu) = \{h_1(\uu,\yy) = 0\}$ for all $\uu\in \PP^2$. By Lemma~\ref{generic_section}, $h_1(\uu, \yy)$ is square-free for generic $\uu\in \PP^2$. This and the fact that for arbitrary $\uu\in X$, hence for a generic $\uu\in \PP^2$, $C(\uu)$ is irreducible of degree $\le 2$ imply that $\deg h_1(\uu,\yy) \le 2$ for a generic $\uu\in \PP^2$, and so $\deg_{\yy} h_1(\xx,\yy) \le 2$.
  \end{proof}
  
  \begin{remark}\label{key_step}
    The key step in the above proof shows the following. If $H\cap(X\times \PP^2)$ is $(2,2)$-grid-free and its defining equation $\tilde{F}(\xx, \yy)\in\CChom[\xx,\yy]$ is factorized as in \eqref{factor2}, then $n = 1$ and $\deg_{\yy}h_1(\xx, \yy) \le 2$.
  \end{remark}
  
  \begin{proof}[Proof of Lemma~\ref{generic_section}]
    Let $d_1, d_2$ be the degrees of $h_1, h_2$ in $\yy$ respectively. Suppose on the contrary that $C_1(\uu)$ and $C_2(\uu)$ share common irreducible components for a generic $\uu\in\PP^2$. So, $h_1(\uu, \yy)$ and $h_2(\uu, \yy)$ have a common divisor in $\CC[\yy]$. Therefore there exist two nonzero polynomials $g_1^{\uu}(\yy) \in \CChom[\yy]$ of degree $<d_2$ and $g_2^{\uu}(\yy) \in \CChom[\yy]$ of degree $<d_1$ such that \begin{equation}\label{res}
      h_1(\uu,\yy)g_1^{\uu}(\yy) + h_2(\uu, \yy)g_2^{\uu}(\yy) = 0.
    \end{equation}
    
    By treating the coefficients of $g_1^{\uu}(\yy)$ and $g_2^{\uu}(\yy)$ as variables, we shall essentially show a nontrivial solution to equation~\eqref{res} exists not only in $\CC$ but also in $\CC[\xx]$. We can view \eqref{res} as a homogeneous system of $M := {{d_1 + d_2 + 1} \choose 2}$ linear equations involving $N:={{d_1 + 1} \choose 2} + {{d_2 + 1}\choose 2}$ variables. Note that the coefficient in the $i$th equation of the $j$th variable, say $c_{ij}$, is a polynomial of $\uu$, that is, $c_{ij} = c_{ij}(\uu)$ for some $c_{ij}(\xx)\in\CC[\xx]$ that depends on $h_1, h_2$ only. Because the system of linear equations has a nontrivial solution and clearly $M > N$, the rank of its coefficient matrix $(c_{ij}(\uu))$ is $<N$. Using the determinants of all $N \times N$ minors of matrix $(c_{ij}(\uu))$, we can rewrite the statement that matrix $(c_{ij}(\uu))$ is of rank $<N$ as $L:={M\choose N}$ polynomial equations of entries in the matrix, say 
    \begin{equation}\label{eq_for_u}
      P_k(c_{ij}(\uu)) = 0, \quad\text{for }k = 1, 2, \dots, L,
    \end{equation} where $P_k(c_{ij}(\xx))$ is a polynomial of $\xx$ independent of $\uu$. Since \eqref{eq_for_u} holds for a generic $\uu\in\PP^2$, we have 
    \begin{equation}
      P_k(c_{ij}(\xx)) = 0 \text{ in }\CC[\xx], \quad\text{for }k = 1, 2, \dots, L.
    \end{equation}
    
    Reversing the argument above, we can deduce that the rank of matrix $(c_{ij}(\xx))$, over the quotient field $\CC(\xx)$, is $<N$, and so there exist two nonzero polynomials $g_1^{\xx}(\yy)\in \CC(\xx)_{\mathrm{hom}}[\yy]$ of degree $<d_2$ and $g_2^{\xx}(\yy)\in \CC(\xx)_{\mathrm{hom}}[\yy]$ of degree $<d_1$ such that \begin{equation}\label{eq_in_cx}
      h_1(\xx,\yy)g_1^{\xx}(\yy) + h_2^{\xx}(\xx, \yy)g_2(\yy) = 0.
    \end{equation}
    Multiplying \eqref{eq_in_cx} by the common denominator of $g_1^{\xx}(\yy)$ and $g_2^{\xx}(\yy)$, we get two nonzero polynomials $g_1(\xx,\yy)\in\CC[\xx,\yy]$ of degree $<d_2$ in $\yy$ and $g_2(\xx,\yy)\in\CC[\xx,\yy]$ of degree $<d_1$ in $\yy$ such that \begin{equation}
      h_1(\xx,\yy)g_1(\xx,\yy) + h_2(\xx,\yy)g_2(\xx,\yy) = 0,
    \end{equation} which is impossible as $\gcdp{h_1, h_2}=1$ and $\deg_{\yy}h_1(\xx,\yy) = d_1 > \deg_{\yy}g_2(\xx,\yy)$.
    
    It remains to prove that $C_1(\uu)$ is reduced for generic $\uu$. Because $h_1(\xx,\yy)\notin \CChom[\xx]$, the polynomial $h_1'(\xx,\yy) := {\partial h_1(\xx,\yy)}/{\partial y_0}$ might be assumed to be nonzero. Again, we assume, on the contrary, that $h_1(\uu,\yy)$ is not square-free for a generic $\uu\in\PP^2$. This implies that $h_1(\uu,\yy)$ and $h_1'(\uu,\yy)$ have a common divisor. The same linear-algebraic argument, applied to $h_1$ and $h_1'$ instead of $h_1$ and $h_2$, then yields a contradiction of the fact that $\gcdp{h_1, h_1'} = 1$.
  \end{proof}
  
  We can apply Theorem~\ref{main_partial_result} to the case when $\PP^2\setminus Y$ is finite, and we obtain a weaker result.
    
  \begin{corollary}
    Suppose $\PP^2\setminus Y = \{\vv_1, \vv_2, \dots, \vv_n\}$.
    If $H\cap(X\times Y)$ is $({2,2})$-grid-free, then either
    \begin{enumerate}[nosep]
      \item there exists $F(\xx,\yy)\in \CChom[\xx,\yy]$ of degree $\le 2$ in $\yy$ such that $H\cap(X\times Y) = \{F = 0\}\cap(X\times Y)$,
      \item or there exists $i\in \{1, 2, \dots, n\}$ such that $\PP^2\times \{\vv_i\} \subset H$.
    \end{enumerate}
  \end{corollary}
  
  \begin{proof}
    Define $D(\vv_i) := \set{\uu\in\PP^2}{(\uu,\vv_i) \in H}$. The second case corresponds to $D(\vv_i) = \PP^2$ for some $i\in\{1,2,\dots, n\}$. Hereafter, we assume that none of those $D(\vv_i)$'s equals $\PP^2$. Note that $D(\vv_i)$ is Zariski-closed for all $i\in[n]$. The set $X' = X\setminus \cup_{i=1}^n D(\vv_i)$ is nonempty Zariski-open subset of $\PP^2$. Because $C(\uu)\subset Y$ for all $\uu \in X'$, we know that $H\cap (X' \times \PP^2) = H\cap (X' \times Y)$ and it is $(2, 2)$-grid-free. Let $H = \{\tilde{F} = 0\}$, and let $f, g$ and $h_1, h_2,\dots, h_n$ be the factors of $\tilde{F}$ as in \eqref{factor2}. By Theorem~\ref{main_partial_result} and Remark~\ref{key_step}, we know that $n = 1$ and $\deg_{\yy}h_1(\xx,\yy)\le 2$.
    
    The set $\{f=0\}\cap X$ is either empty or infinite. So, for $H\cap(X\times Y)$ to be $(2,2)$-grid-free we must have $\variety{f=0}\cap X=\emptyset$. Similarly, we know that $\variety{g = 0} \cap Y = \emptyset$. These imply that the zeros of $F(\xx, \yy) := h_1(\xx, \yy)$ in $X\times Y$ agree with $\tilde{F}$.
  \end{proof}
    
\section{Fields of finite characteristic}\label{sec_finite_char}

  A standard model-theoretic argument allows us to transfer statements over fields of characteristic $0$ to the fields of large characteristic.
  
  \begin{theorem}\label{mt}
    Let $\phi$ be a sentence in the language of rings. The following are equivalent. \begin{enumerate}[nosep]
      \item $\phi$ is true in complex numbers.
      \item $\phi$ is true in every algebraically closed field of characteristic zero.
      \item $\phi$ is true in all algebraically closed fields of characteristic $p$ for all sufficiently large prime $p$.
    \end{enumerate}
  \end{theorem}
  
  The theorem, which can be found as a part of {\cite[Corollary 2.2.10]{MR1924282}}, is an application of the compactness theorem and the completeness of the theory of algebraically closed field of fixed characteristic. We refer the readers to \cite[Section 2.1]{MR1924282} for further details of the theorem and related notions.
  
  As quantifiers over all polynomials are not part of the language of rings, one has to limit the degree of hypersurface $H$ and the complexity of the open set $X$ in Theorem~\ref{main_partial_result}. We now formulate the analog over the fields of large characteristic.
  
  \begin{theorem}\label{finite}
    Let $\KK$ be an algebraically closed field of large characteristic, let $H$ be a hypersurface in $\PP^2(\KK)\times \PP^2(\KK)$ of bounded degree, and let $X$ be a Zariski-open subset of $\PP^2(\KK)$ of bounded complexity (i.e. $X$ is a Zariski-open subset of $\PP^2(\KK)$ that can be described by some first-order predicate in the language of rings of bounded length). If $H$ is $(2,2)$-grid-free in $X\times \PP^2(\KK)$, then there exists $F(\xx, \yy)\in\KKhom[\xx,\yy]$ of degree $\le 2$ in $\yy$ such that $H\cap(X\times \PP^2) = \variety{F = 0}\cap(X\times \PP^2)$.
  \end{theorem}
  
  The proof essentially rewrites Theorem~\ref{main_partial_result} as a sentence in the language of rings to which Theorem~\ref{mt} is applicable.
  
  \begin{proof}
    The language of rings, denoted by $\Lr$, consists of binary function symbols $+, -, \cdot$ and constant symbols $0, 1$. Let the hypersurface $H$ be $\{h_{\cc}(\xx, \yy) = 0\}$, where $h$ is the term in $\Lr$ so that $h_{\cc}(\xx, \yy)$ is the polynomial of degree $d_x, d_y$ in $\xx, \yy$ with coefficients $\cc$, and similarly let the Zariski-open set $X$ be described by the first-order predicate $\chi_{\dd}(\xx)$ in $\Lr$, where $\dd$ represents the complex numbers appear in the description. We rewrite parts of Theorem~\ref{main_partial_result} as follows.
    \begin{enumerate}[nosep]
      \item ``$X$ is nonempty'' translates to $\exists \uu\ \chi_{\dd}(\uu)$;
      \item ``$H \cap (X \times \PP^2)$ is $(2, 2)$-grid-free'' translates to $\neg \big[\exists \uu_0\exists \uu_1\exists \vv_0\exists\vv_1. h_{\cc}(\uu_0, \vv_0)=0 \wedge h_{\cc}(\uu_0, \vv_1)=0 \wedge h_{\cc}(\uu_1, \vv_0)=0 \wedge h_{\cc}(\uu_1, \vv_1)=0 \wedge \chi_{\dd}(\uu_0) \wedge \chi_{\dd}(\uu_1)\big]$;
      \item ``$H\cap (X \times \PP^2) = \{F = 0\} \cap (X\times \PP^2)$'' translates to $\forall\uu\forall\vv. \chi_{\dd}(\uu)\to (h_{\cc}(\uu, \vv) = 0 \leftrightarrow F(\uu, \vv) = 0)$.
    \end{enumerate}
    
    Note that the proof of Theorem~\ref{main_partial_result} also shows that $F(\xx, \yy)$ is of degree $\le d_x$ in $\xx$. Hence one can rewrite ``there exists $F(\xx, \yy)\in\CChom[\xx, \yy]$ of degree $\le 2$ in $\yy$ (and of degree $\le d_x$ in $\xx$)'' as a bunch of existential quantifiers over the coefficients of $F$ and a first-order predicate on these coefficients stating that $F(\xx, \yy)$ is homogeneous separately in $\xx$ and $\yy$.
    
    Assembling these parts together, we can rewrite Theorem~\ref{main_partial_result} as a first-order formula $\psi_{h, \chi}(\cc, \dd)$ in $\Lr$ with free variables $\cc$ and $\dd$. Thus its universal closure $\phi_{h, \chi} := \forall \cc \forall \dd. \psi_{h, \chi}(\cc, \dd)$ is true in $\CC$. Let $D\in \NN$ be the bound on the degree of $H$ and the complexity of $X$ in Theorem~\ref{finite}. Denote the set of the first-order terms $\Lr$ of degree $d_x, d_y \le D$ in $\xx, \yy$ by $T_{\mathrm{hom}}$, and the set of the first-order predicates describing Zariski-open sets of length $\le D$ by $P_{\mathrm{open}}$. Note that both $T_{\mathrm{hom}}$ and $P_{\mathrm{open}}$ are finite. Finally we apply Theorem~\ref{mt} to the sentence $\phi := \bigwedge_{h\in T_{\mathrm{hom}}}\bigwedge_{\chi\in P_{\mathrm{open}}} \phi_{h, \chi}$.
  \end{proof}
  
\section{Acknowledgement}
  The second author would like to thank Hong Wang for her suggestion on the choice of terminology in algebraic geometry. We thank the referees for suggestions that helped to improve the exposition.

\bibliographystyle{jalpha}
\bibliography{alg_graph_k2t}

\end{document}